\documentclass[10pt,a4paper]{amsart} 
\usepackage{geometry}
\usepackage{tabls}
\usepackage{url}
\usepackage[utf8]{inputenc}
\usepackage{amsfonts}
\usepackage{amssymb}
\usepackage{mathrsfs}
\usepackage{amsmath}
\usepackage{hyperref}
\usepackage{amscd}
\usepackage{fancyhdr}
\usepackage{enumerate}
\usepackage{paralist}
\usepackage{xypic}
\usepackage{graphicx}
\usepackage{cite}
\usepackage{lipsum}
\usepackage{footmisc}
\usepackage[all,cmtip]{xy}

\allowdisplaybreaks

\numberwithin{equation}{section}
\theoremstyle{plain}
\newtheorem{theorem}{Theorem}[section]
\newtheorem{lemma}[theorem]{Lemma}
\newtheorem{remark}[theorem]{Remark}
\numberwithin{equation}{section}

\def\a{\alpha}
\def\b{\beta}

\def\L{\Lambda}

\def\W{\Omega}

\def\ni{\noindent}

\def\rank{{\rm rank\/}}

\def\ind{{\rm ind\/}}
\def\Stab{{\rm Stab\/}}

\def\End{\mathrm{End}}

\def\R{\mathbb{R} }

\def\fg{\mathfrak{g}}

\begin{document}
	
	\title{On the general part of the perturbed Vafa-Witten moduli spaces on 4-manifolds}

	%\authorrunning{Short form of author list} % if too long for running head
	
%	\institute{
%		Ren Guan \at
%		School of Mathematics and Statistics, 
%		Jiangsu Normal University, Xuzhou 221100, China\\
%		\email{guanren@jsnu.edu.cn} 
%	}
	\author{Ren Guan}
	\address{School of Mathematics and Statistics, 
		Jiangsu Normal University, Xuzhou 221100, China}
	\email{guanren@jsnu.edu.cn}

	\begin{abstract}
		In this paper we consider the   Vafa-Witten equations on closed, oriented and smooth 4-manifolds, and construct a set of perturbation terms to establish the transversality of the perturbed   Vafa-Witten equations at the general part of the solutions. Then we show that for a generic choice of the perturbation terms, this part of the moduli space for the structure group $SU(2)$ or $SO(3)$ on a closed 4-manifold is a smooth manifold of dimension zero.

	\end{abstract}

	

	\maketitle
	
	\tableofcontents
	
	\section{Introduction}
	
	The celebrated Donaldson's polynomial invariants \cite{D1,D2} and Seiberg-Witten invariants \cite{W} are of great importance to low dimensional topology, especially in the research of  the smooth structures of 4-manifolds. For example, we can use both invariants to show that $K3\#\overline{\mathbb{C}\mathrm{P}^2}$ and $\#3\mathbb{C}\mathrm{P}^2\#20\overline{\mathbb{C}\mathrm{P}^2}$ are homeomorphic but not diffeomorphic (see \cite[Theorem A]{D1}, \cite[\S 2, \S 4]{W}). Both invariants are defined by the moduli spaces of solutions to certain equations arised from gauge theory. Donaldson's polynomial invariants are defined via the anti-self-dual (ASD) Yang-Mills equations on 4-manifolds with non-abelian structure group \cite{D1,D2},  and Seiberg-Witten invariants rely on the $U(1)$ monopole equations \cite{Mor,W}.
	
	When study an $N=4$ topologically twisted supersymmetric Yang-Mills theory on 4-manifolds, Vafa and Witten proposed a pair of new gauge equation \cite{VW}, and nowadays we call it \emph{Vafa-Witten equation}.  Vafa-Witten equation has the following form:
	\begin{equation}\left\{
		\begin{aligned}
			&d_A^*B +d_A C=0,\\&F_A^++\frac{1}{8}[B\centerdot B]+\frac{1}{2}[B,C]=0,  \\
		\end{aligned}\right. \label{aa}
	\end{equation}
	where $A$ is a connection on a principal $G$-bundle $P$ over a smooth 4-manifold $X$, $B$ is a self-dual 2-form with values in the adjoint bundle $\fg_P$, and $C$ is a section of $\fg_P$. On closed 4-manifolds, Vafa-Witten equation \eqref{aa} is equivalent to
	\begin{equation}
		\left\{
		\begin{aligned}
			&d_AC=d_A^*B=0\\&F_A^++\frac{1}{8}[B\centerdot B]=[B,C]=0.\\
		\end{aligned}\right.
		\label{ab}
	\end{equation}
	When $G=SU(2)$ or $SO(3)$ and the connection $A$ is irreducible, $d_AC=0$ implies $C=0$ (cf.  \cite[\S 2.1]{Ma}), and then the equation \eqref{ab} is further reduced  to
	\begin{equation}\left\{
		\begin{aligned}
			&d_A^*B=0\\&F_A^++\frac{1}{8}[B\centerdot B]=0.\\
		\end{aligned}\right.
		\label{ac}
	\end{equation}
	We call equation \eqref{ac}  the \emph{reduced Vafa-Witten equation}.
	
	Since Vafa-Witten equation is gauge-equivariant \cite[\S1.3]{Ma}, we can define the \emph{Vafa-Witten moduli space}:
	$$\mathcal{M}_{VW}=\{(A,B,C)\in\mathcal{C}_k(P):\mathcal{VW}(A,B,C)=0\}/\mathcal{G}_{k+1}(P),$$
	i.e., the set of solutions to the equation \eqref{aa} modulo the gauge transformation (See \S2 for details). In general, $\mathcal{M}_{VW}$
	is very complicated because the transversality often fails. According to the values of $B$ and $C$ on the underlying manifold $X$, $\mathcal{M}_{VW}$ can be  divided into three parts:
	$$\mathcal{M}_{VW}=\mathcal{M}_{\text{asd}}\cup\mathcal{M}_{RVW}\cup\widetilde{\mathcal{M}}_{VW}$$
	where
	\begin{align*}
		\mathcal{M}_{\text{asd}}&:=\{[A,B,C]\in\mathcal{M}_{VW}:\text{$B\equiv 0$ and $C\equiv0$ on $X$}\},\\
		\mathcal{M}_{RVW}&:=\{[A,B,C]\in\mathcal{M}_{VW}:\text{$B\not\equiv 0$ and $C\equiv0$ on $X$}\},\\
		\widetilde{\mathcal{M}}_{VW}&:=\{[A,B,C]\in\mathcal{M}_{VW}:\text{$C\not\equiv0$ on $X$}\}.
	\end{align*}
	The space $\mathcal{M}_{\text{asd}}$ is the well-known moduli space of the anti-self-dual (ASD) connections, and has been deeply studied by Donaldson, Uhlenbeck and others\cite{D1,D2,FU}. We call $\mathcal{M}_{RVW}$ the \emph{reduced} part of $\mathcal{M}_{VW}$ and $\widetilde{\mathcal{M}}_{VW}$ the \emph{general} part.

	The Lie group used by Donaldson's invariants is the noncommutative Lie group $SU(2)$, and the moduli spaces are defined by the ASD (anti-self-dual) equations $F_A^+=0$; The Lie group used by Seiberg-Witten invariants is the commutative Lie group $U(1)$, and the Seiberg-Witten invariants depend not only on the connections $A$ but also on the additional fields $\Phi$ \cite{Mor,SW1}. Due to the commutativity of $U(1)$ and \emph{a priori} boundness of the additional field $\Phi$ \cite[\S5.2]{Mor}, the transversality of Seiberg-Witten equations and the compactness of the moduli spaces can be established more easily than the Donaldson's case, and that makes Seiberg-Witten invariants easier to construct and more applicable to the study of differential structures of smooth 4-manifolds. 
	
	The Vafa-Witten equations studied in this paper are also defined on the $SU(2)$- and $SO(3)$-principal bundles. Mares said that if the transversality is established, then the dimensions of the Vafa-Witten moduli spaces are zero \cite[Section 1.1.3]{Ma}, which do not desided by the geometric structure of bundles and the underlying manifolds. Perhaps the invariants defined by Vafa-Witten equations have wider applications.
	
	Before defining the invariants, Donaldson, Seiberg and Witten construct perturbations to the corresponding equations first (see \cite{D1,Mor} for details), the main reason for doing this is that the moduli spaces of solutions to the original equations are not smooth manifolds due to the lack of transversality. This can be settled down by constructing proper perturbations and consider the moduli spaces of solutions to the perturbed equations. After a lot of sophisticated analysis, the invariants are constructed and can be applied to distinguish the differential structures of topologically homeomorphic 4-manifolds.

	In  \cite{Tan}, by applying the method of \cite{Fe}, Tanaka constructs suitable perturbations of Vafa-Witten equations on closed symplectic 4-manifolds, then he shows that if the structure group is $SU(2)$ or $SO(3)$, the moduli spaces are smooth manifolds of dimension zero for generic choices of the perturbation parameters. In this paper, we construct perturbations and establish the transversality of the perturbed   Vafa-Witten equations on closed, oriented and smooth Riemannian 4-manifolds at the general part of the solutions, i.e., the set of solutions $(A,B,C)$ where $A$ is irreducible and $C\not\equiv 0$, then we show that the dimensions of this part of the moduli spaces are zero for  generic perturbations. Our main result is the following theorem.
	
	\begin{theorem}\label{main}Let $(X,g)$ be a closed, oriented and smooth Riemannian 4-manifold, $r>k\geq3$ are two integers. Then there is a first-category subset $\mathcal{T}^r_{fc}\subset\mathcal{T}^r:=C^r(GL(\Lambda^{1}))\times C^r(GL(\Lambda^{2,+}))\times C^r(GL(\Lambda^{2,+}))\times C^r(X,\Lambda^{1})\times C^r(X,\Lambda^{2,+})$ such that for all $(\tau^1,\tau^2,\tau^3,\theta,\gamma)$ in $\mathcal{T}^r-\mathcal{T}^r_{fc}$, the moduli space of the $L_k^2$ solutions $[A,B,C]$  to the equation
	\begin{equation}\left\{
			\begin{aligned}
				&d_A^*B+d_AC+\tau^1\big((B+[B,C])\centerdot\theta+C\otimes\theta\big)=0\\
				&F_A^++\frac{1}{8}[B\centerdot B]+\frac{1}{2}\tau^2[B,C]+\tau^3B+C\otimes\gamma=0\\
			\end{aligned}\right.\end{equation}
	with $A$ irreducible and $C\not\equiv0$ is a smooth manifold of dimension zero.
	\end{theorem}
	
	Note that the Lie algebra of $SU(2)$ and $SO(3)$ are isomorphic: $\mathfrak{su}(2)\cong\mathfrak{so}(3)$, so in the following, $\mathfrak{so}(3)$ is replaced by $\mathfrak{su}(2)$  wherever it appears.
	
	\section{Vafa-Witten equations and their perturbations}
	
	In this section we lay out the general set-up of Vafa-Witten equations and construct the perturbations for establishing the transversality. 
	
	\subsection{The Vafa-Witten map}
	
	Let $(X,g)$ be a closed, oriented and smooth Riemannian 4-manifold, $P\to X$ a principle $SU(2)$- or $SO(3)$-bundle on $X$ and $\mathcal{A}_k (P)$ the space of $L_k^2$ connections where $k\geq 3$ is an integer. $\mathcal{A}_k (P)$ is an affine space, so for any fixed $L_k^2$ connection $A_0$, we have
	$$\mathcal{A}_k (P)=A_0+L_k^2(X,\mathfrak{su}(2)_P\otimes\Lambda^1).$$
	Then we define the \emph{configuration spaces} and \emph{target spaces}
	$$\begin{aligned}
		\mathcal{C}_k(P)&:=\mathcal{A}_k (P)\times L_k^2(X,\mathfrak{su}(2)_P\otimes\Lambda^{2,+})\oplus L_k^2(X, \mathfrak{su}(2)_P),\\
		\mathcal{C}_{k-1}'(P)&:=L_{k-1}^2(X,\mathfrak{su}(2)_P\otimes\Lambda^1)\oplus L_{k-1}^2(X,\mathfrak{su}(2)_P\otimes\Lambda^{2,+}).
	\end{aligned}$$
	It's easy to see that $\mathcal{C}_k(P)$ is also an affine vector space, which means for any fixed $L_k^2$ triple $(A_0,B_0,C_0)\in\mathcal{C}_k(P)$, we have
	$$\mathcal{C}_k(P)=(A_0,B_0,C_0)+L_k^2(X,\mathfrak{su}(2)_P\otimes\Lambda^1)\oplus L_k^2(X,\mathfrak{su}(2)_P\otimes\Lambda^{2,+})\oplus
	L_k^2(X,\mathfrak{su}(2)_P).$$
	
	The \emph{Vafa-Witten map} is defined as
	\begin{equation}\begin{aligned}
			\mathcal{VW}&:\mathcal{C}_k(P)\to\mathcal{C}_{k-1}'(P),\\
			\mathcal{VW}(A,B,C)&:=\begin{pmatrix}
				d_A^*B+d_AC\\
				F_A^++\frac{1}{8}[B\centerdot B]+\frac{1}{2}[B,C]\\
			\end{pmatrix}.
		\end{aligned} \label{cc}
	\end{equation}
	The action of an $L_{k+1}^2$ gauge transformation $\zeta\in\mathcal{G}_{k+1}(P)$ on any $(A,B,C)\in\mathcal{C}_k(P)$ is given by
	$$\zeta\cdot(A,B,C):=(A-(d_A\zeta^{-1})\zeta,\zeta^{-1}B\zeta,\zeta^{-1}C\zeta),$$
	and a connection $A$ is called \emph{irreducible} if $\Stab(A):=\{\zeta\in\mathcal{G}_{k+1}(P):\zeta\cdot A=A\}=Z(G)$, the center of $G$. 
	
	It's not hard to see that
	$$\mathcal{VW}(\zeta\cdot(A,B,C))=\begin{pmatrix}
		\zeta^{-1}(d_A^*B+d_AC)\zeta\\
		\zeta^{-1}(F_A^++\frac{1}{8}[B\centerdot B]+\frac{1}{2}[B,C])\zeta\\
	\end{pmatrix},$$
	so the map $\mathcal{VW}$ is gauge-equivariant (cf.  \cite[\S 3.2.1]{Ma}). Define the quotient space $\mathcal{B}_k (P):=\mathcal{C}_k (P)/\mathcal{G}_{k+1}(P)$ (cf. \cite[Propostion 2.8]{FL},  \cite[Theorem 3.2.3]{Ma}).
	
	Denote by $\mathcal{C}_k^\diamond (P)$ the triples $(A,B,C)\in\mathcal{C}_k(P)$ where $A$ is irreducible and $C\not\equiv 0$, and we call it the \emph{general} part of $\mathcal{C}_k(P)$. The slice theorem (cf.  \cite[Propostion 2.8]{FL} \cite[Theorem 3.2.3]{Ma}) implies that the quotient space $\mathcal{B}_k^\diamond (P):=\mathcal{C}_k^\diamond (P)/\mathcal{G}_{k+1}(P)\subset\mathcal{B}_k(P)$ is an open and smooth Hilbert submanifold of $\mathcal{B}_k(P)$. 
	
	\subsection{The perturbed   Vafa-Witten map}
	
	We now construct the \emph{perturbed   Vafa-Witten map} as follows:
\begin{equation}\begin{aligned}
		\mathcal{VW}&:\mathcal{T}^r\times\mathcal{C}_k(P)\to\mathcal{C}_{k-1}'(P),\\
		\mathcal{VW}(\tau^1,\tau^2,\tau^3,\theta,\gamma,A,B,C)&:=\begin{pmatrix}
			d_A^*B+d_AC+\tau^1\big((B+[B,C])\centerdot\theta+C\otimes\theta\big)\\
			F_A^++\frac{1}{8}[B\centerdot B]+\frac{1}{2}\tau^2[B,C]+\tau^3B+C\otimes\gamma
		\end{pmatrix}, \label{cd}
	\end{aligned}
\end{equation}
where $\mathcal{T}^r:=C^r(GL(\Lambda^{1}))\times C^r(GL(\Lambda^{2,+}))\times C^r(GL(\Lambda^{2,+}))\times C^r(X,\Lambda^{1})\times C^r(X,\Lambda^{2,+})$ is the Banach manifold of $C^r$ perturbation parameters $(\tau^1,\tau^2,\tau^3,\theta,\gamma)$ (with $r$ large enough, say $r>k$). The gauge group $\mathcal{G}_{k+1}(P)$ acts trivially on the space of perturbations $\mathcal{T}^r$, so the map $\mathcal{VW}$ is also gauge-equivariant, and $\mathcal{PM}_k(P):=\mathcal{VW}^{-1}(0)/\mathcal{G}_{k+1}(P)\subset\mathcal{T}^r\times\mathcal{B}_k(P)$ is the parametrized moduli space of the perturbed   Vafa-Witten equation. Let $\mathcal{PM}_k^\diamond(P)=\mathcal{PM}_k(P)\cap(\mathcal{T}^r\times\mathcal{B}_k^\diamond(P))$. Note that the gauge-equivariant map $\mathcal{VW}$ defines a section of a Banach vector bundle $\overline{\mathcal{E}}_k$ over $\mathcal{T}^r\times\mathcal{B}_k^\diamond(P)$ with total space $\overline{\mathcal{E}}_k:=(\mathcal{T}^r\times\mathcal{C}_k^\diamond(P))\times_{\mathcal{G}_{k+1}(P)}\mathcal{C}'_{k-1}(P)$. In particular, the parametrized moduli space $\mathcal{PM}_k^\diamond(P)$ is the zero set of the section $\mathcal{VW}(\cdot,\cdot)$ of the vector bundle $\overline{\mathcal{E}}_k$ over $\mathcal{T}^r\times\mathcal{B}_k^\diamond(P)$. 
	
	\subsection{The Kuranishi complex}
	
	Recall that the deformation complex associated to the ASD connection $A$ is given by
	$$0\to\Omega^0(X,\mathfrak{g}_P)\xrightarrow{d_A}\Omega^1(X, \mathfrak{g}_P)\xrightarrow{d^+_A}\Omega^{2,+}(X, \mathfrak{g}_P)\to 0.$$
	The first differential $d_A$ is the infinitesimal action of the gauge transformation and the second differential $d_A^+$ is the linearization of the equation $F_A^+=0$. The cohomology groups $H_A^\bullet$  each have a geometric interpretation: $H_A^0=\mathrm{Ker}~d_A$ is zero if and only if $A$ is irreducible, $H^1_A=\mathrm{Ker}~d^+_A/\mathrm{Im}~d_A$ is the formal tangent space $T_{[A]}\mathcal{M}_{ASD}$, here $[A]$ denotes the equivalent class of $A$ under the gauge transformation, $H_A^2=\mathrm{Coker}~d_A^+$ is zero if and only if $d_A^+:\Omega^1(X, \mathfrak{g}_P)\to\Omega^{2,+}(X, \mathfrak{g}_P)$ is surjective, or equivalently,  the map $A\mapsto F_A^+$ vanishes transversely at $A$. The sum of $d_A^*:\Omega^1(X,\mathfrak{g}_P)\to\Omega^0(X,\mathfrak{g}_P)$ and $d_A^+:\Omega^1(X,\mathfrak{g}_P)\to\Omega^{2,+}(X, \mathfrak{g}_P)$:
	$$d_A^*+d_A^+:\Omega^1(X, \mathfrak{g}_P)\to\Omega^0(X, \mathfrak{g}_P)\oplus\Omega^{2,+}(X,\mathfrak{g}_P)$$
	is elliptic, which is crucial to the construction of Donaldson's polynomial  invariants.
	
	For each chosen perturbation parameters $(\tau^1,\tau^2,\tau^3,\theta,\gamma)$, the associated deformation complex of the perturbed   Vafa-Witten map is
	$$\begin{aligned}
		0\to& L^2_{k+1}(X,\Lambda^0\otimes\mathfrak{su}(2)_P)\xrightarrow{d^0_{(A,B,C)}}L^2_k (X,\mathfrak{su}(2)_P\otimes\Lambda^1)\oplus L^2_k(X,\mathfrak{su}(2)_P\otimes\Lambda^{2,+})\oplus\\
		&L^2_k(X,\Lambda^0\otimes\mathfrak{su}(2)_P)\xrightarrow{d^1_{(A,B,C)}} L^2_{k-1}(X,\mathfrak{su}(2)_P\otimes\Lambda^1)\oplus L^2_{k-1}(X,\mathfrak{su}(2)_P\otimes\Lambda^{2,+})\to 0,
	\end{aligned}$$
	where
	\begin{equation}
		d^0_{(A,B,C)}(\xi)=(d_A\xi, [B, \xi],[C, \xi]) 
		\label{cf} \end{equation}
	is the linearization of the action of gauge group, and
	\begin{equation} \begin{aligned}
			&d^1_{(A,B,C)}(a,b,c)=(D\mathcal{VW})_{(A,B,C)}(a,b,c) \\
			&=\begin{pmatrix}
				d^*_Ab+d_Ac-[B\centerdot a]-[C,a]+\tau^1\big((b+[b,C]+[B,c])\centerdot\theta+c\otimes\theta\big)\\
				d_A^+a+\frac{1}{4}[b\centerdot B] +\frac12 \tau^2[b,C] +\frac12 \tau^2[B,c]+\tau^3 b+c\otimes\gamma
			\end{pmatrix}
		\end{aligned} \label{cg}
	\end{equation}
	is the linearization of the perturbed   Vafa-Witten maps. 
	Also we have (cf. \cite{Fe, Ma})
	$$\begin{aligned}
		d^1_{(A,B,C)}\circ d^0_{(A,B,C)}(\xi)=\begin{pmatrix}
			[\xi,d_A^*B+d_AC+\tau^1\big((B+[B,C])\centerdot\theta+C\otimes\theta\big)]\\
			[\xi,F_A^++\frac{1}{8}[B\centerdot B]+\frac12\tau^2 [B,C]+\tau^3B+C\otimes\gamma]\end{pmatrix}.
	\end{aligned}$$
	Therefore $d^1_{(A,B,C)}\circ d^0_{(A,B,C)}=0$ if and only if $\mathcal{VW}(\tau^1,\tau^2,\tau^3,\theta,\gamma,A,B,C)=0$, i.e. the sequence is a complex if and only if $[A,B,C]\in\widetilde{\mathcal{M}}_{P,VW}(\tau^1,\tau^2,\tau^3,\theta,\gamma):=\{[A,B,C]\in\mathcal{B}_k(P):\mathcal{VW}(\tau^1,\tau^2,\tau^3,\theta,\gamma,A,B,C)=0\}$.
	
	The $L^2$ adjoint of $d^0_{(A,B,C)}$ is (cf. \cite{Ma})
	$$d^{0,*}_{(A,B,C)}(a,b,c)=d_A^*a+[b\cdot B]+[c,C],$$
	and the combined opertor
	$$\begin{aligned}
		\mathcal{D}_{(A,B,C)}&=d^{1}_{(A,B,C)}+ d^{0,*}_{(A,B,C)}:\begin{matrix} L^2_k(X,\mathfrak{su}(2)_P\otimes\Lambda^1) \\ \oplus \\ L^2_k(X,\mathfrak{su}(2)_P\otimes\Lambda^{2,+}) \\ \oplus \\ L^2_k (X, \mathfrak{su}(2)_P)\end{matrix}\to\begin{matrix} L^2_{k-1}(X,\mathfrak{su}(2)_P\otimes\Lambda^1) \\ \oplus \\ L^2_{k-1}(X,\mathfrak{su}(2)_P\otimes\Lambda^{2,+}) \\ \oplus \\ L^2_{k-1}(X,\mathfrak{su}(2)_P)\end{matrix}
	\end{aligned}$$
	differs from the following operator 
	\begin{equation} \mathcal{D}_A =\left(
		\begin{array}{ccc}
			0 & d_A^* & d_A \\
			d_A^+ & 0 & 0 \\
			d_A^* & 0 & 0 \\
		\end{array}
		\right):\begin{pmatrix} a \\ b \\ c \end{pmatrix} \mapsto
		\begin{pmatrix} d_A^* b +d_A c \\ d_A^+ a \\ d_A^* a \end{pmatrix}
		\label{cj} \end{equation}
	by zeroth-order terms. The Sobolev multiplication theorem and the Rellich embedding theorem imply that the Sobolev multiplication map $L_k^2\times L_k^2\to L_k^2$ is continuous and the inclusion $L_k^2\hookrightarrow L_{k-1}^2$ is compact when $k\geq 3$, so we have $\ind(\mathcal{D}_{(A,B,C)})=\ind(\mathcal{D}_A)=0$, the second equality is derived from the self-adjointness of elliptic operator $\mathcal{D}_A$. 
	
	The above complex is an elliptic deformation complex for the perturbed   Vafa-Witten equation with cohomology groups
	$$H_{(A,B,C)}^0:=\mathrm{Ker}d^{0}_{(A,B,C)},H_{(A,B,C)}^1:=\mathrm{Ker}d^{1}_{(A,B,C)}/\mathrm{Im}~d^{0}_{(A,B,C)},H_{(A,B,C)}^2:=\mathrm{Coker}d^{1}_{(A,B,C)}.$$
	Similarly, $H_{(A,B,C)}^0$ is the Lie algebra of the stabilizer of the triple $(A,B,C)$,  and $H_{(A,B,C)}^0=0$ if the stabilizer of $(A,B,C)$ is $Z(G)$, which means $A$ is irreducible, $H_{(A,B,C)}^1$ is the tangent space $T_{[A,B,C]}\widetilde{\mathcal{M}}_{P,VW}(\tau^1,\tau^2,\tau^3,\theta,\gamma)$ and if  $H_{(A,B,C)}^2=0$, then $\mathrm{Coker}(D\mathcal{VW})_{(A,B,C)}=0$ and $[A,B,C]$ is a regular point of $\widetilde{\mathcal{M}}_{P,VW}(\tau^1,\tau^2,\tau^3,\theta,\gamma)$. We will prove later that every point of  $\widetilde{\mathcal{M}}_{P,VW}^*(\tau^1,\tau^2,\tau^3,\theta,\gamma):=\widetilde{\mathcal{M}}_{P,VW}(\tau^1,\tau^2,\tau^3,\theta,\gamma)\cap\mathcal{B}_k^\diamond(P)$ is regular and hence $\widetilde{\mathcal{M}}_{P,VW}^*(\tau^1,\tau^2,\tau^3,\theta,\gamma)$ is a smooth manifold of dimension $\ind(\mathcal{D}_A)=0$, which is precisely the Theorem \ref{main} says.

	\section{Quadratic expansion of the perturbed   Vafa-Witten map}
	
	We now find the quadratic expansion of $\mathcal{VW}(\tau^1,\tau^2,\tau^3,\theta,\gamma,A,B,C)$ with a fixed perturbation parameter to build the regularity result (see  \cite[\S3]{FL} for the perturbed $PU(2)$-monopole equations and  \cite[\S3.2.2]{Ma} for Vafa-Witten equations). 
	Let $(A_0,B_0,C_0)\in\mathcal{C}_k(P)$ be a fixed smooth triple and $(a,b,c)\in L_k^2(X,\mathfrak{su}(2)_P\otimes\Lambda^1)\oplus L_k^2(X,\mathfrak{su}(2)_P\otimes\Lambda^{2,+})\oplus L_k^2(X,\mathfrak{su}(2)_P)$, then
	\begin{align*}
		&\mathcal{VW}(\tau^1,\tau^2,\tau^3,\theta,\gamma,A_0+a,B_0+b,C_0+c)\\
		=&\begin{pmatrix}d_{A_0+a}^*(B_0+b)+d_{A_0+a}(C_0+c)+\tau^1\big((B_0+b+[B_0+b, C_0+c])\centerdot\theta+(C_0+c)\wedge\theta\big)\\
			F_{A_0+a}^++\frac{1}{8}[B_0+b\centerdot B_0+b]+\frac{1}{2}\tau^2[B_0+b,C_0+c]+\tau^3(B_0+b)+(C_0+c)\otimes\gamma\end{pmatrix}\\
		=&\mathcal{VW}(\tau^1,\tau^2,\tau^3,\theta,\gamma,A_0,B_0,C_0)+d_{(A_0,B_0,C_0)}^1(a,b,c)\\
		&+\begin{pmatrix}-[a\centerdot b]+[a,c]+\tau^1([b, c]\centerdot\theta)\\
			\frac{1}{2}[a\wedge a]^++\frac{1}{8}[b\centerdot b]+\frac{1}{2}\tau^2[b,c]\end{pmatrix} \\
		=&\mathcal{VW}(\tau^1,\tau^2,\tau^3,\theta,\gamma,A_0,B_0,C_0)+d_{(A_0,B_0,C_0)}^1(a,b,c)+\{(a,b,c),(a,b,c)\},
	\end{align*}
	where
	$$\begin{aligned}
		\{(a,b,c),(a,b,c)\}:=\begin{pmatrix}-[a\centerdot b]+[a,c]+\tau^1([b,c]\centerdot\theta)\\
			\frac{1}{2}[a\wedge a]^++\frac{1}{8}[b\centerdot b]+\frac{1}{2}\tau^2[b,c]\end{pmatrix}.
	\end{aligned}$$
	Given $(A_0,B_0,C_0)\in\mathcal{C}_k(P)$ and $(u_0,v_0)\in\mathcal{C}_{k-1}'(P)$, consider the inhomogeneous equation
	\begin{equation}\label{pvw}
		\mathcal{VW}(\tau^1,\tau^2,\tau^3,\theta,\gamma,A_0+a,B_0+b,C_0+c)=(u_0,v_0)
	\end{equation}
	for the triplets $(a,b,c)$. To make the equation elliptic, we impose the gauge-fixing condition
	\begin{equation}\label{gf}
		d^{0,*}_{(A,B,C)}(a,b,c)=w,
	\end{equation}
	then combine \eqref{pvw} and \eqref{gf}, we get an elliptic equation:
	\begin{equation} \mathcal{D}_{(A_0,B_0,C_0)}(a,b,c)+\{(a,b,c),(a,b,c)\}=(w,u,v)
		\label{ccc}
	\end{equation}
	where $(u,v)=(u_0,v_0)-\mathcal{VW}(\tau^1,\tau^2,\tau^3,\theta,\gamma,A_0,B_0,C_0)$. Equation \eqref{ccc} is called \emph{the perturbed   Vafa-Witten equation under Coulomb gauge}.
	
	Mimicking the regular results of  \cite[\S3]{FL} and  \cite[\S3.3]{Ma}, we have
	\begin{theorem} \label{tj} {\bf (Global estimate for $L_1^2$ solutions to the inhomogeneous perturbed   Vafa-Witten plus Coulomb slice equations, cf.  \cite[Corollary 3.4]{FL}, \cite[Theorem 3.3.1]{Ma})\/}
		Let $(X,g)$ be a closed, oriented and smooth Riemannian 4-manifold, $P\to X$ a principle $SU(2)$- or $SO(3)$-bundle on $X$ and let $(A_0,B_0,C_0)$ be a $C^{\infty}$ configuration in $\mathcal{C}(P)$. Then there is a positive constant $\epsilon=\epsilon(A_0,B_0,C_0)$ such that if $(a,b,c)$ is an $L_1^2$ solution to the  equation \eqref{ccc}, where $(w,u,v)$ is in $L_k^2$ and $||(a,b,c)||_{L^4(X)}<\epsilon$, and $k\geq 3$ is an integer, then $(a,b,c)\in L_{k+1}^2$ and there is a polynomial $Q_k(x,y)$, with positive real coefficients, depending at most on $(A_0,B_0,C_0),k$ such that $Q_k(0,0)=0$ and
		$$||(a,b,c)||_{L_{k+1,A_0}^2(X)}\leq Q_k\Big(||(w,u,v)||_{L_{k,A_0}^2(X)},||(a,b,c)||_{L^2(X)}\Big).$$
		In particular, if $(w,u,v)$ is in $C^r$ then $(a,b,c)$ is in $C^{r+1}$. 
	\end{theorem}
	
	\begin{theorem} \label{tk}
		{\bf  (Global regularity of $L_k^2$ solutions to the perturbed   Vafa-Witten equations for $k\geq 2$, cf.  \cite[Proposition 3.7]{FL}, \cite[Theorem 3.3.2]{Ma})} Let $(X,g)$ be a closed, oriented and smooth Riemannian 4-manifold and $P\to X$ a principle $SU(2)$- or $SO(3)$-bundle on $X$. Let $k\geq 3$ be an integer and suppose that $(A,B,C)$ is an $L_k^2$ solution to $\mathcal{VW}(\tau^1,\tau^2,\tau^3,\theta,\gamma,A,B,C)=0$ for fixed $C^r$ perturbation parameters $(\tau^1,\tau^2,\tau^3,\theta,\gamma)$, then there is a gauge transformation $\zeta\in L_{k+1}^2(\mathcal{G}_P)$ such that $\zeta\cdot (A,B,C)$ is $C^\infty$ over $X$.
	\end{theorem}
	
	\section{Transversality of the general part of the moduli spaces}
	
	In this section we show that the perturbed   Vafa-Witten map \eqref{cd}, when viewed as a section of the Banach vector bundle $\overline{\mathcal{E}}_k$ over $\mathcal{T}^r\times\mathcal{B}_k^\diamond(P)$, is transverse to the zero section of $\overline{\mathcal{E}}_k$. 
	
	\subsection{The linearization of the perturbed Vafa-Witten map}
	
	For $(\tau^1,\tau^2,\tau^3,\theta,\gamma)\in\mathcal{T}^r$, if $\Gamma:=(\tau^1,\tau^2,\tau^3,\theta,\gamma,A,B,C) \in \mathcal{T}^r\times\mathcal{B}_k(P)$ satisfies $\mathcal{VW}(\Gamma)=0$, then the linearization of the map $\mathcal{VW}$ at $\Gamma$ is 
	\begin{align*}
		&(D\mathcal{VW})_{\Gamma}(\delta\tau^1,\delta\tau^2,\delta\tau^3,\delta\theta,\delta\gamma,a,b,c)\\
		=&\begin{pmatrix}(D\mathcal{VW}_1)_{\Gamma}(\delta\tau^1,\delta\tau^2,\delta\tau^3,\delta\theta,\delta\gamma,a,b,c)\\
			(D\mathcal{VW}_2)_{\Gamma}(\delta\tau^1,\delta\tau^2,\delta\tau^3,\delta\theta,\delta\gamma,a,b,c)\end{pmatrix}\\
		=&\begin{pmatrix}
			d^*_Ab +d_A c-[B\centerdot a] -[C,a]+\delta\tau^1\big((B+[B,C])\centerdot\theta+C\wedge\theta\big)+\tau^1\big((b+[b,C]+[B,c])\centerdot\theta+c\wedge\theta\big)\\
			+\tau^1\big((B+[B,C])\centerdot\delta\theta+C\wedge\delta\theta\big)\\
			d_A^+a+\frac{1}{4}[b\centerdot B] +\frac12 \delta\tau^2[B,C]+\frac12 \tau^2[B,c] +\frac12\tau^2 [b, C]+\delta\tau^3B+\tau^3b+c\otimes\gamma+C\otimes\delta\gamma
		\end{pmatrix},
	\end{align*}
	where $(\delta\tau^1,\delta\tau^2,\delta\tau^3,\delta\theta,\delta\gamma,a,b,c)\in\mathcal{T}^r\times T_{[A,B,C]}\widetilde{\mathcal{M}}_{P,VW}^*(\tau^1,\tau^2,\tau^3,\theta,\gamma)$ (It's easy to see that $T_{(\tau^1,\tau^2,\tau^3,\theta,\gamma)}\mathcal{T}^r=\mathcal{T}^r$).   
	Note that  $(D\mathcal{VW})_{\Gamma}$ differs from $d^1_{A,B,C}$ in \eqref{cg} by a bounded linear operator acting on $\left(\delta\tau^1,\delta\tau^2,\delta\tau^3,\delta\theta,\delta\gamma\right)$. This fact, together with the estimate in Theorem~\ref{tj}, implies that $(D\mathcal{VW})_{\Gamma}$ has closed range. Hence
	$$\mathrm{Ran}(D\mathcal{VW})_{\Gamma} \not= \mathcal{C}'_{k-1}(P)$$
	if and only if there is a nonzero pair $(\phi,\psi)\in\mathcal{C}'_{k-1}(P)=L_{k-1}^2(X,\mathfrak{su}(2)_P\otimes\Lambda^1)\oplus L_{k-1}^2(X,\mathfrak{su}(2)_P\otimes\Lambda^{2,+})$ such that $\forall(\delta\tau^1,\delta\tau^2,\delta\tau^3,\delta\theta,\delta\gamma,a,b,c)\in\mathcal{T}^r\times T_{[A,B,C]}\widetilde{\mathcal{M}}_{P,VW}^*(\tau^1,$ $\tau^2,\tau^3,\theta,\gamma)$ we have
	\begin{equation}\label{yy}\begin{aligned}
			&\langle(D\mathcal{VW})_{\Gamma}(\delta\tau^1,\delta\tau^2,\delta\tau^3,\delta\theta,\delta\gamma,a,b,c),(\phi,\psi)\rangle_{L^2(X)}\\
			=&\langle(D\mathcal{VW}_1)_{\Gamma}(\delta\tau^1,\delta\tau^2,\delta\tau^3,\delta\theta,\delta\gamma,a,b,c),\phi\rangle_{L^2(X)}\\
			+&\langle(D\mathcal{VW}_2)_{\Gamma}(\delta\tau^1,\delta\tau^2,\delta\tau^3,\delta\theta,\delta\gamma,a,b,c),\psi\rangle_{L^2(X)}=0.
	\end{aligned}\end{equation}
	The above formula also implies $(\phi,\psi)\in\mathrm{Ker}(D\mathcal{VW})_{\Gamma}^*$ (where $(D\mathcal{VW})_{\Gamma}^*$ is the $L^2(X)$ adjoint operator of $(D\mathcal{VW})_{\Gamma}$), then the elliptic regularity for the Laplacian $(D\mathcal{VW})_{\Gamma}(D\mathcal{VW})_{\Gamma}^*$ with $C^{r-1}$ coefficients implies that $(\phi,\psi)$ is $C^{r+1}$(cf.  \cite[\S5]{FL}). Also, the Aronszajn's theorem (cf.  \cite[Remark 3]{Ar}, \cite[Theorem 1.8]{Kaz}) implies that $(\phi,\psi)\in\mathrm{Ker}\left((D\mathcal{VW})_{\Gamma}(D\mathcal{VW})_{\Gamma}^*\right)$ has the unique continuation property (cf.  \cite[Lemma 5.9]{FL}). Therefore, to prove that $(\phi,\psi)\equiv 0$ on $X$, it is only necessary to prove that $(\phi,\psi)$ is zero on an open subset of $X$.
	
	\subsection{The Agmon-Nirenberg unique continuation theorem}
	
	As in the case of the ASD equations \cite[Lemma 4.3.21]{D2} and $PU(2)$-monopole equations \cite[\S5.3.2]{FL}, our proof of the unique continuation property for zero points $[A,B,C]$ of the perturbed   Vafa-Witten map \eqref{cd} in radial gauge relies also on the Agmon and Nirenberg's unique continuation theorem for the solutions to an ordinary differential equation on
	a Hilbert space  \cite{Ag}. We state the version of that theorem we need here:
	\begin{theorem}\label{agth}
		Let $(\mathfrak{H},(\cdot,\cdot))$ be a Hilbert space and let $\mathcal{P}:\mathrm{Dom}(\mathcal{P}(r))=\mathfrak{H}_D\subset\mathfrak{H}\to\mathfrak{H}$ be a family of symmetric linear operators for $r\in[r_0,R)$. Suppose that $\eta\in C^1([r_0,R),\mathfrak{H})$ with $\eta(r)\in\mathfrak{H}_D$ and $\mathcal{P}\eta\in C^0([r_0,R),\mathfrak{H})$ such that
		$$\left|\left|\frac{d\eta}{dr}-\mathcal{P}(r)\eta(r)\right|\right|\leq c_1||\eta(r)||$$
		for some positive constants $c_1$ and all $r\in[r_0,R)$. If the function $r\mapsto(\eta(r),\mathcal{P}(r)\eta(r))$ is differentiable for $r\in[r_0,R)$ and satisfies
		$$\left|\left|\frac{d\mathcal{P}}{dr}\eta\right|\right|\leq c_2(||\mathcal{P}\eta||+||\eta||)$$
		for positive constants $c_2$ and every $r\in[r_0,R)$, then the following holds: If $\eta(r)=0$ for $r_0\leq r\leq r_1<R$, then $\eta(r)=0$ for all $r\in[r_0,R)$.
	\end{theorem}
	
	\subsection{Establishment of the transversality}
	
	Before establishing the transversality, we prove the following two lemmas first: The first lemma establishes the unique continuation property of the connection $A$ when $[A,B,C]$ is a zero point of the perturbed   Vafa-Witten map \eqref{cd}, the second lemma provides a full rank part of the first component of the map \eqref{cd}, which is crucial to the establishment of the transversality.
	
	\begin{lemma}\label{xzz}Let $(X,g)$ be an oriented, closed and smooth Riemannian 4-manifold and $P\to X$ a principle $SU(2)$- or $SO(3)$-bundle on $X$.  For every perturbation parameter $(\tau^1,\tau^2,\tau^3,\theta,\gamma)\in\mathcal{T}^r$, if $[A,B,C]\in\mathcal{B}_k(P)$ is a solution to the perturbed   Vafa-Witten equation
		\begin{equation}\label{yz}\left\{
			\begin{aligned}
				&d_A^*B+d_AC+\tau^1\big((B+[B,C])\centerdot\theta+C\otimes\theta\big)=0\\
				&F_A^++\frac{1}{8}[B\centerdot B]+\frac{1}{2}\tau^2[B,C]+\tau^3B+C\otimes\gamma=0\\
			\end{aligned}\right.\end{equation}
		such that $C\not\equiv 0$ and $[B,C]\equiv 0$ on $X$, then $A$ is reducible on $X$.
	\end{lemma}
	
	\begin{proof}
		For a solution $[A,B,C]\in\mathcal{B}_k(P)$ to \eqref{yz} such that $C\not\equiv 0$ and $[B,C]\equiv 0$ on $X$, the first equation of \ref{xzz} is reduced to
		$$d_A^*B+d_AC+\tau^1(B\centerdot\theta+C\otimes\theta)=0.$$
		For $(\lambda,\mu)\in L_{k}^2(X,\mathfrak{su}(2)_P\otimes\Lambda^{2,+})\oplus L_{k}^2(X,\mathfrak{su}(2)_P)$, define the linear map 
		$$\mathcal{L}_{A,\tau^1,\theta}(\lambda,\mu):=d_A^*\lambda+d_A\mu+\tau^1(\lambda\centerdot\theta+\mu\otimes\theta)\in L_{k-1}^2(X,\mathfrak{su}(2)_P\otimes\Lambda^1),$$
		then $\mathcal{L}_{A,\tau^1,\theta}(B,C)=0$ and also, $\mathcal{L}_{A,\tau^1,\theta}^*\mathcal{L}_{A,\tau^1,\theta}(B,C)=0$ where $\mathcal{L}_{A,\tau^1,\theta}^*$ is the $L^2$ adjoint of $\mathcal{L}_{A,\tau^1,\theta}$. So the Aronszajn's theorem (cf.  \cite[Remark 3]{Ar}, \cite[Theorem 1.8]{Kaz}) implies that $(B,C)$ has the unique continuation property. Let $X_{B,C}:=\{x\in X:B(x)\neq 0\}\cap\{x\in X:C(x)\neq 0\}\subseteq X$, then $X_{B,C}$ is either $X$ or an open dense subset of $X$.	Then	the property of Lie algebra $\mathfrak{su}(2)$ implies $B$ is at most rank 1 on $X$, so there is $\xi\in \Omega^0(X_{B,C},\mathfrak{su}(2)_P)$ with $\langle\xi,\xi\rangle=1$ and $\omega\in L_{k}^2(X,\Lambda^{2,+})$ such that $B=\xi\otimes\omega$ (cf. \cite[\S4.2]{Ma}) and $C=\langle C,C\rangle^\frac12\xi=|C|\xi$.  In addition, we have $[B\centerdot B]\equiv 0$ on $X$, so the equation \eqref{yz} can be reduced  to
		\begin{equation}\label{yz2}\left\{
			\begin{aligned}
				&d_A^*B+d_AC+\tau^1(B\centerdot\theta+C\otimes\theta)=0\\
				&F_A^++\tau^3B+C\otimes\gamma=0.\\
			\end{aligned}\right.\end{equation}
				
		We have
		\begin{equation}
			\begin{aligned}\label{yz3}
				0&=d_A^*B+d_AC+\tau^1(B\centerdot\theta+C\otimes\theta)\\
				&=d_A^*(\xi\otimes\omega)+d_A(|C|\xi)+\tau^1(\xi\otimes\omega\centerdot\theta+|C|\xi\otimes\theta)\\
				&=-d_A\xi\centerdot\omega+\xi\otimes d^*\omega+|C|d_A\xi+\xi\otimes d|C|+\xi\otimes\tau^1(\omega\centerdot\theta+|C|\theta).\\
			\end{aligned}
		\end{equation}
		$\langle\xi,\xi\rangle=1$ implies $\langle d_A\xi,\xi\rangle=0$, take inner product with $\xi$ we get $d^*\omega+d|C|+\tau^1(\omega\centerdot\theta+|C|\theta)=0$, so $-d_A\xi\centerdot\omega+|C|d_A\xi=0$ on $X_{B,C}$ and Lemma~\ref{xxx} implies $d_A\xi=0$ on $X_{B,C}$ (cf.  \cite[\S4.2]{Ma}). That's means $A$ is reducible on $X_{B,C}$.
		
		Next we will show that in fact $A$ is reducible on $X$. The method used here is analogous to the cases of the ASD equations \cite[Lemma 4.3.21]{D2} and $PU(2)$-monopole equations \cite[\S5.3.2]{FL}.
		
		Choose a point $x_0\in X_{B,C}$ and let $\rho$ be the injectivity radius of $X$ at $x_0$, there is a positive number $0<\epsilon<\frac{1}{2}\rho$ such that the geodesic ball $B(x_0,\epsilon)\subset X_{B,C}$. We will show that $A$ is reducible on $B(x_0,2\epsilon)$. We trivialize $\mathfrak{su}(2)_P$ over $B(x_0,2\epsilon)-\{x_0\}$ using parallel transport along radial geodesics (cf.  \cite[\S2.3.1]{D2}), this gives an isomorphism 
		\begin{equation}\label{is}
			\mathfrak{su}(2)_P|_{B(x_0,\epsilon)-\{x_0\}}\cong\mathfrak{su}(2)_{x_0}\times S^3\times(0,2\epsilon)
		\end{equation}
		where $\mathfrak{su}(2)_{x_0}\cong\mathfrak{su}(2)$. Note that $A$ is in radial gauge with respect to the point $x_0$, so the radial component of connection $A$ in this trivialization is zero. We let $\mathfrak{A}:=\mathfrak{A}(r)$, $r\in(0,2\epsilon)$, denote the resulting one-parameter family of the connection on the bundle $\mathfrak{su}(2)_{x_0}\times S^3$ over $S^3$. A section $B$ of the bundle $L^2_k(X,\mathfrak{su}(2)_P\otimes\Lambda^{2,+})$ over $B(x_0,2\epsilon)-\{x_0\}$ pulls back, via the isomorphism $\Lambda^{2,+}\otimes\mathfrak{su}(2)_P|_{B(x_0,2\epsilon)-\{x_0\}}\cong\Lambda^{2,+}\otimes\mathfrak{su}(2)_{x_0}\times S^3\times(0,2\epsilon)$, to a one-parameter family of sections $\mathfrak{B}(r)$ of the bundle $\Lambda^{2,+}\otimes\mathfrak{su}(2)_{x_0}\times S^3$ over $S^3$.
		
		Under the isomorphism \eqref{is}, $\xi$ can be viewed as a map to the structure group $SU(2)$ or $SO(3)$ and it satisfies
		$$0=\langle d_A\xi,\frac{\partial}{\partial r}\rangle=\frac{\partial\xi}{\partial r},$$
		hence we can extend $\xi$ by parallel
		translation via $A$ along radial geodesics emanating from $x_0$ to a gauge transformation $\hat{\xi}$ on connections on the bundle $\mathfrak{su}(2)_{x_0}\times S^3\times(0,2\epsilon)$. 
		
		Let $\hat{A}:=\hat{\xi}^{-1}\cdot A=A-(d_A\hat{\xi})\hat{\xi}^{-1}$, note that $d_A\hat{\xi}=0$ on $B(x_0,\epsilon)$, we have $\hat{A}=A$ on $B(x_0,\epsilon)$. Under the isomorphism $S^3\times(0,2\epsilon)\cong B(x_0,2\epsilon)-\{x_0\}$, the metric $g$ on $B(x_0,2\epsilon)-\{x_0\}$ pulls back to
		$$g=dr^2+g_r=dr^2+\gamma(r,\theta)d\theta^2,$$
		where $g_r:=\gamma(r,\theta)d\theta^2$ is the metric on $S^3$ pulled back from the restriction $g|_{S^3(x_0,r)}$ to the geodesic sphere $S^3(x_0,r):=\{x\in X:d_g(x,x_0)=r\}$. Denote by $*_{g_r}$ the Hodge star operator for the metric $g_r$ on $S^3$ and $*_g$ the Hodge star operator for the metric $g$ on $X$. For any differential form $\eta$ on $S^3$, we have 
		\begin{align*}
			*_g\eta&=dr\wedge*_{g_r}\eta,\\
			*_{g_r}\eta&=*_g(dr\wedge\eta).
		\end{align*}
		
		Let $\{e^1,e^2,e^3\}$ be an oriented, orthonormal frame for $T^*S^3$, then $\{e^0:=dr,e^1,e^2,e^3\}$ is an oriented orthonormal basis for $T^*X$ over $B(x_0,2\epsilon)-\{x_0\}$. Note that the linear space of self-dual $\R$-valued two forms on $X$ is 3-dimensional $C(X)$-linear space and $\R$-valued two forms on $S^3$ is 3-dimensional $C(S^3)$-linear space, hence a endomorphism $\tau^3\in C^r(\End_\R(\Lambda^{2,+}(X)))$ can also be regarded as an endomorphism of $L_k^2(S^3,\Lambda^2)$.
		
		With all above understood, $\{e^0\wedge e^1+e^2\wedge e^3,e^0\wedge e^2+e^3\wedge e^1,e^0\wedge e^3+e^1\wedge e^2\}$ is an orthonormal basis of $\Lambda^{2,+}(B(x_0,2\epsilon))$ (cf.  \cite[\S4.1.1]{Ma}). Let $A=\sum_{i=0}^3A_ie^i$ and $F_A=\sum_{0\leq i<j\leq 3}F_{ij}e^i\wedge e^j$ be the linear representation in the corresponding basis. The radial component of connection $A$ is zero implies $A_0=0$ on $B(x_0,2\epsilon)$ and we have $F_{0j}=\frac{dA_j}{dr}$, hence
		\begin{align*}
			F_A^+=&\frac{1}{2}(1+*_g)F_A\\
			=&\frac{1}{2}(F_{01}+F_{23})(e^0\wedge e^1+e^2\wedge e^3)+\frac{1}{2}(F_{02}+F_{31})(e^0\wedge e^2+e^3\wedge e^1)\\
			&+\frac{1}{2}(F_{03}+F_{12})(e^0\wedge e^3+e^1\wedge e^2)\\
			=&\frac{1}{2}e^0\wedge(F_{01}e^1+F_{02}e^2+F_{03}e^3)+\frac{1}{2}(F_{23}e^2\wedge e^3+F_{31}e^3\wedge e^1+F_{12}e^1\wedge e^2)\\
			=&\frac{1}{2}dr\wedge\frac{d\mathfrak{A}(r)}{dr}+\frac{1}{2}F_{\mathfrak{A(r)}}.
		\end{align*}
		
		Let $\tau^3=\begin{pmatrix}
			\tau^3_{11} & \tau^3_{12} & \tau^3_{13}\\
			\tau^3_{21} & \tau^3_{22} & \tau^3_{23}\\
			\tau^3_{31} & \tau^3_{32} & \tau^3_{33}\\
		\end{pmatrix}$ be the representation matrix under the basis $\{e^0\wedge e^1+e^2\wedge e^3,e^0\wedge e^2+e^3\wedge e^1,e^0\wedge e^3+e^1\wedge e^2\}$, it's not hard to see that the representation matrix of $\tau^3$ under the basis $\{e^2\wedge e^3, e^3\wedge e^1,e^1\wedge e^2\}$ of $L_k^2(S^3,\Lambda^2)$ is the same one. Let
		$$\omega=B_1(e^0\wedge e^1+e^2\wedge e^3)+B_2(e^0\wedge e^2+e^3\wedge e^1)+B_3(e^0\wedge e^3+e^1\wedge e^2)$$
		where $B_i\in L_k^2(B(x_0,2\epsilon),\R)$, $i=1,2,3$, then 
		$$\mathfrak{B}(r)=\xi\otimes(B_1e^2\wedge e^3+B_2e^3\wedge e^1+B_3e^1\wedge e^2)$$
		and we have
		\begin{align*}
			\tau^3B=&\tau^3(\xi\otimes\omega)\\
			=&\xi\otimes\big(B_1\tau^3(e^0\wedge e^1+e^2\wedge e^3)+B_2\tau^3(e^0\wedge e^2+e^3\wedge e^1)+B_3\tau^3(e^0\wedge e^3+e^1\wedge e^2)\big)\\
			=&\xi\otimes\Big(B_1\big(\tau^3_{11}(e^0\wedge e^1+e^2\wedge e^3)+\tau^3_{12}(e^0\wedge e^2+e^3\wedge e^1)+\tau^3_{13}(e^0\wedge e^3+e^1\wedge e^2)\big)+\\
			&B_2\big(\tau^3_{21}(e^0\wedge e^1+e^2\wedge e^3)+\tau^3_{22}(e^0\wedge e^2+e^3\wedge e^1)+\tau^3_{23}(e^0\wedge e^3+e^1\wedge e^2)\big)+\\
			&B_3\big(\tau^3_{31}(e^0\wedge e^1+e^2\wedge e^3)+\tau^3_{32}(e^0\wedge e^2+e^3\wedge e^1)+\tau^3_{33}(e^0\wedge e^3+e^1\wedge e^2)\big)\Big)\\
			=&\xi\otimes e^0\wedge\big((B_1\tau^3_{11}+B_2\tau^3_{21}+B_3\tau^3_{31})e^1+(B_1\tau^3_{12}+B_2\tau^3_{22}+B_3\tau^3_{32})e^2\\
			&+(B_1\tau^3_{13}+B_2\tau^3_{23}+B_3\tau^3_{33})e^3\big)+\xi\otimes\big((B_1\tau^3_{11}+B_2\tau^3_{21}+B_3\tau^3_{31})e^2\wedge e^3\\
			&+(B_1\tau^3_{12}+B_2\tau^3_{22}+B_3\tau^3_{32})e^3\wedge e^1+(B_1\tau^3_{13}+B_2\tau^3_{23}+B_3\tau^3_{33})e^1\wedge e^2\big)\\
			=&dr\wedge*_{g_r}\tau^3\mathfrak{B}(r)+\tau^3\mathfrak{B}(r),\\
			C\otimes\gamma=&C\otimes\big(\gamma_1(e^0\wedge e^1+e^2\wedge e^3)+\gamma_2(e^0\wedge e^2+e^3\wedge e^1)+\gamma_3(e^0\wedge e^3+e^1\wedge e^2)\big)\\
			=&\mathfrak{C}(r)\otimes\big(dr\wedge*_{g_r}\gamma(r)+\gamma(r)\big).
		\end{align*}
		So $F_A^++\tau^3B+C\otimes\gamma=0$ is equivalent to
		\begin{align*}
			dr\wedge\left(\frac{1}{2}\frac{d\mathfrak{A}(r)}{dr}+*_{g_r}\tau^3\mathfrak{B}(r)+\mathfrak{C}(r)\otimes*_{g_r}\gamma(r)\right)+\left(\frac{1}{2}F_{\mathfrak{A(r)}}+\tau^3\mathfrak{B}(r)+\mathfrak{C}(r)\otimes\gamma(r)\right)=0,
		\end{align*}
		which means
		\begin{equation}
			\left\{
			\begin{aligned}
				\frac{1}{2}\frac{d\mathfrak{A}(r)}{dr}+*_{g_r}\tau^3\mathfrak{B}(r)+\mathfrak{C}(r)\otimes*_{g_r}\gamma(r)&=0\\
				\frac{1}{2}F_{\mathfrak{A(r)}}+\tau^3\mathfrak{B}(r)+\mathfrak{C}(r)\otimes\gamma(r)&=0\\
			\end{aligned}\right.
		\end{equation}
		and hence on $B(x_0,2\epsilon)$,
		\begin{equation}\label{asd3}
			\frac{d\mathfrak{A}(r)}{dr}=*_{g_r}F_{\mathfrak{A(r)}}.
		\end{equation}
		$\hat{\mathfrak{A}}(r)$ and $\mathfrak{A}(r)$ are both solutions to \eqref{asd3}, and $\hat{\mathfrak{A}}(r)=\mathfrak{A}(r)$ when $0<r\leq\epsilon$. Let $a(t):=\hat{\mathfrak{A}}(r)-\mathfrak{A}(r)$, then $a(r)=0$ when $0<r\leq\epsilon$ and
		\begin{equation}
			\begin{aligned}\label{ode}
				\frac{da}{dr}&=*_{g_r}\left(d\hat{\mathfrak{A}}(r)-d\mathfrak{A}(r)+\frac{1}{2}[\hat{\mathfrak{A}}(r),\hat{\mathfrak{A}}(r)]-\frac{1}{2}[\mathfrak{A}(r),\mathfrak{A}(r)]\right)\\
				&=*_{g_r}\left(da+\frac{1}{2}[\mathfrak{A},a]+\frac{1}{2}[a,\hat{\mathfrak{A}}]\right).
			\end{aligned}
		\end{equation}
		Hence
		$$\left|\left|\frac{da}{dr}-*_{g_r}da\right|\right|_{L_k^2}\leq c_1||a||_{L_k^2}$$
		for some constant $c_1$ depending on $\hat{\mathfrak{A}}(r)$ and $\mathfrak{A}(r)$. 
		
		On $\Omega^1(S^3)$, $*_{g_r}^2=(-1)^{1*(3-1)}=1$. Hence for any two 1-forms $a_1,a_2\in L_k^2((S^3,g_r),\Lambda^1)$, we have
		\begin{align*}
			&\int_{S^3}\langle a_1,*_{g_r}da_2\rangle d\mathrm{vol}_r\\
			=&\int_{S^3}\langle a_1,*_{g_r}d*_{g_r}^2a_2\rangle d\mathrm{vol}_r=\int_{S^3}\langle a_1,(-1)^{3(2+1)+1}d^{*_{g_r}}*_{g_r}a_2\rangle d\mathrm{vol}_r\\
			=&\int_{S^3}\langle da_1,*_{g_r}a_2\rangle d\mathrm{vol}_r=\int_{S^3}\langle*_{g_r}da_1,a_2\rangle d\mathrm{vol}_r,
		\end{align*}
		which means the operators $*_{g_r}d$ are self-adjoint with respect to the metrics $g_r$, $r\in(0,2\epsilon)$.
		Let $d\mathrm{vol}$ be the volume form on $S^3$ defined by the standard metric $g_{std}$, then
		$$d\mathrm{vol}_r=h_r^2d\mathrm{vol},~r\in(0,2\epsilon)$$
		for some positive function $h_r$ on $S^3$. Define the operators $\mathcal{Q}_r:=h_r\circ(*_{g_r}d)\circ h_r^{-1}$ for every $r\in(0,2\epsilon)$ and Hilbert-space isomorphisms $L_k^2((S^3,g_r),\Lambda^1)\to L_k^2((S^3,g_{std}),\Lambda^1)$ by $a\to\alpha:=h_ra$, then
		\begin{align*}
			&\int_{S^3}\langle \mathcal{Q}_r(\alpha_1),\alpha_2\rangle d\mathrm{vol}=\int_{S^3}\langle h_r\circ(*_{g_r}d)\circ h_r^{-1}h_ra_1,h_ra_2\rangle d\mathrm{vol}\\
			=&\int_{S^3}\langle *_{g_r}da_1,a_2\rangle d\mathrm{vol}_r=\int_{S^3}\langle a_1,*_{g_r}da_2\rangle d\mathrm{vol}_r\\
			=&\int_{S^3}\langle h_ra_1,h_r\circ(*_{g_r}d)\circ h_r^{-1}h_ra_2\rangle d\mathrm{vol}=\int_{S^3}\langle \alpha_1,\mathcal{Q}_r(\alpha_2)\rangle d\mathrm{vol},
		\end{align*}
		therefore $\mathcal{Q}_r$ is self-adjoint with respect to $g_{std}$ for every $r\in(0,2\epsilon)$ and has dense domain $L_1^2((S^3,g_{std}),\Lambda^1)$. Since $a=h_r^{-1}\alpha$, we have
		$$\frac{da}{dr}=-h_r^{-2}\frac{dh_r}{dr}\alpha+h_r^{-1}\frac{d\alpha}{dr}$$
		and subsituting into \eqref{ode} gives
		\begin{equation}
			\begin{aligned}
				-h_r^{-2}\frac{dh_r}{dr}\alpha+h_r^{-1}\frac{d\alpha}{dr}&=h_r^{-1}\mathcal{Q}_rh_ra+\frac{1}{2}*_{g_r}([\mathfrak{A},h_r^{-1}\alpha]+[h_r^{-1}\alpha,\hat{\mathfrak{A}}]),\\
				\frac{d\alpha}{dr}&=\mathcal{Q}_r\alpha+\frac{1}{2}*_{g_r}([\mathfrak{A},\alpha]+[\alpha,\hat{\mathfrak{A}}])+h_r^{-1}\frac{dh_r}{dr}\alpha
			\end{aligned}
		\end{equation}
		for $r\in(0,2\epsilon)$. On compact interval $[\frac{1}{2}\epsilon,2\epsilon]$, the continuous function $h_r^{-1}\frac{dh_r}{dr}$ is bounded and hence there is a constant $c_1$ independent of $a$ such that
		$$\left|\left|\frac{d\alpha}{dr}-\mathcal{Q}_r\alpha\right|\right|_{L_k^2}\leq c_1||\alpha||_{L_k^2},~\frac{1}{2}\epsilon\leq r\leq2\epsilon.$$
		Also, the definition of $\mathcal{Q}_r$ yields the pointwise bounds
		$$\left|\left(\frac{d\mathcal{Q}_r}{dr}a\right)(r)\right|\leq c'(|(\nabla a)(r)|+|a(r)|)$$
		for $\frac{1}{2}\epsilon\leq r\leq2\epsilon$ and some constant $c'$ independent of $a$, where $\nabla$  denotes covariant derivatives on $\Lambda^{1}\otimes\mathfrak{su}(2)_{x_0}\times S^3$ over $S^3$ which is independent of $r$. Thus the standard elliptic estimate for $\mathcal{Q}_r$ implies
		$$\left|\left|\frac{d\mathcal{Q}_r}{dr}a\right|\right|_{L_k^2}\leq c_2(||\mathcal{Q}_ra||_{L_k^2}+||a||_{L_k^2})$$
		for $\frac{1}{2}\epsilon\leq r\leq2\epsilon$ and some constant $c_2$ independent of $a$. The conditions of Theorem \ref{agth} have been verified and $a(r)=0$ when $r\in(0,\epsilon)$, so we must have $a(r)=0$ for $r\in(0,2\epsilon)$, which means $d_A\hat{\xi}=0$ on $B(x_0,2\epsilon)$.
		
		By applying the above argument to every point $x_0$ of the open dense subset $X_{B,C}\subseteq X$, we can see that the extension $\hat{\xi}$ of $\xi$ can be extended to all of $X$ and $d_A\hat{\xi}=0$ on $X$, hence $A$ is reducible on $X$.
	\end{proof}
	
	\begin{lemma}\label{xz}Let $(X,g)$ be an oriented, closed and smooth Riemannian 4-manifold and $P\to X$ a principle $SU(2)$- or $SO(3)$-bundle on $X$. Then there is a first-category subset $\mathcal{T}^r_1\subset\mathcal{T}^r$ such that for every perturbation parameter $(\tau^1,\tau^2,\tau^3,\theta,\gamma)\in\mathcal{T}^r-\mathcal{T}^r_1$, if $[A,B,C]\in\mathcal{B}_k^\diamond(P)$ is a solution to the perturbed   Vafa-Witten equation
		\eqref{yz}, then $(B+[B,C])\centerdot\theta+C\otimes\theta$ is rank 3 on some open subset $U\subset X$.
	\end{lemma}

	\begin{proof}
		Let $\mathcal{T}^r_1:=\{(\tau^1,\tau^2,\tau^3,\theta,\gamma)\in\mathcal{T}^r:\theta=0 \text{ on some open subset of } X\}$, it's not hard to see that $\mathcal{T}^r_1$ is a nowhere dense, hence first-category subset of $\mathcal{T}^r$. For any fixed $(\tau^1,\tau^2,\tau^3,\theta,\gamma)\in\mathcal{T}^r-\mathcal{T}^r_1$, 
		let $[A,B,C]\in\mathcal{B}_k^\diamond(P)$ be any solution to \eqref{yz}, the definition of $\mathcal{B}_k^\diamond(P)$ implies that $A$ is irreducible on $X$, hence by Lemma \ref{xzz} we have $[B,C]\not\equiv0$ on $X$, so there is an open subset $U\subset X$ such that $[B,C]\neq 0$ on $U$ and Lemma \ref{xxx} implies  $(B+[B,C])\centerdot\theta+C\otimes\theta$ is rank 3 on $U$.
	\end{proof}
	
	Now we prove that $(\phi,\psi)$ is zero on an open subset of $X$. For a perturbation parameter $(\tau^1,\tau^2,\tau^3,\theta,\gamma)\in\mathcal{T}^r-\mathcal{T}^r_1$, let $[A,B,C]\in\mathcal{B}_k^\diamond(P)$ be a solution to the corresponding perturbed   Vafa-Witten equation \eqref{yz}. Lemma \ref{xzz} implies that there is an open subset $V\subset X$ such that $[B,C]\neq 0$ on $V$. In equations \eqref{yy}, set $(\delta\tau^1,\delta\tau^3,\delta\theta,\delta\gamma,a,b,c)=0$ we have
	$$\langle\delta\tau^2[B,C],\psi\rangle_{L^2(X)}=0,~\forall\delta\tau^2\in C^r(\mathfrak{gl}(\Lambda^{2,+})),$$
	similiarly,
	$$\langle\delta\tau^3B,\psi\rangle_{L^2(X)}=\langle C\otimes\delta\gamma,\psi\rangle_{L^2(X)}=0,~\forall(\delta\tau^3,\delta\gamma)\in C^r(\mathfrak{gl}(\Lambda^{2,+}))\times C^r(X,\Lambda^1).$$
	On $V$, $[B,C]\neq 0$, so applying Lemma \ref{xx} and \cite[Lemma 2.3]{Fe}, we have $\psi\equiv 0$ on $V$. Then set $(\delta\tau^2,\delta\tau^3,\delta\theta,\delta\gamma,a,b,c)=0$ we can get
	$$\langle\delta\tau^1\big((B+[B,C])\centerdot\theta+C\otimes\theta\big),\phi\rangle_{L^2(X)}=0,~\forall\delta\tau^1\in C^r(\mathfrak{gl}(\Lambda^1))$$
	on $V$, Lemma \ref{xz} and  \cite[Lemma 2.3]{Fe} imply $\phi\equiv 0$ on $V$.
	
	Hence  $(\phi,\psi)\equiv 0$ on $V$, and $(\phi,\psi)\equiv 0$ on $X$ by applying the unique continuation property of the Laplacian $(D\mathcal{VW})_{\Gamma}(D\mathcal{VW})_{\Gamma}^*$. We finish the establishment of the transversality.
	
	By the slice result (cf.  \cite[Proposition 2.8]{FL}), $T_{[A,B,C]}\mathcal{B}_k^\diamond(P)$ may be identified with $\mathrm{Ker}~d_{(A,B,C)}^{0,*}$. For $(a,b,c)\in\mathrm{Ker}~d_{(A,B,C)}^{0,*}$, we have
	$$\begin{aligned}
		(D\mathcal{VW})_{(A,B,C)}(0,1,0,0,0,a,b,c)&=d_{(A,B,C)}^1(a,b,c)\\
		&=(d_{(A,B,C)}^{0,*}+d_{(A,B,C)}^1)(a,b,c),\end{aligned}$$
	which makes the differential $(D\mathcal{VW})_{(A,B,C)}|_{\{0\}\times T\mathcal{C}_k^\diamond(P))}$ a Fredholm operator, where ${\{0\}\times T\mathcal{C}_k^\diamond(P)}=T((\tau^1,\tau^2,\tau^3,\theta,\gamma)\times\mathcal{B}_k^\diamond(P))$. Thus $\mathcal{VW}$ is a Fredholm section when restricted to the fixed-parameter fibers $(\tau^1,\tau^2,\tau^3,\theta,\gamma)\times\mathcal{C}_k^\diamond(P)\subset(\mathcal{T}^r-\mathcal{T}^r_1)\times\mathcal{C}_k^\diamond(P)$ where $(\tau^1,\tau^2,\tau^3,\theta,\gamma)\in\mathcal{T}^r-\mathcal{T}^r_1$, and the Sard-Smale theorem (cf. \cite[Proposition 4.12]{Fe}, \cite[Proposition 4.3.11]{D2}) implies that there is a first-category subset $\mathcal{T}^r_2\subset\mathcal{T}^r-\mathcal{T}^r_1$ such that the zero sets in $\mathcal{C}_k^\diamond(P)$ of $\mathcal{VW}(\tau^1,\tau^2,\tau^3,\theta,\gamma,\cdot)$ are regular for all perturbations $(\tau^1,\tau^2,\tau^3,\theta,\gamma)\in\mathcal{T}^r-(\mathcal{T}^r_{1}\cup\mathcal{T}^r_{2})$. 
	
	In summary, we have the following theorem, which is also Theorem \ref{main}:
	
	\begin{theorem}Let $(X,g)$ be a closed, oriented and smooth Riemannian 4-manifold, Then there is a first-category subset $\mathcal{T}^r_{fc}\subset\mathcal{T}^r$ such that for all $(\tau^1,\tau^2,\tau^3,\theta,\gamma)$ in $\mathcal{T}^r-\mathcal{T}^r_{fc}$ the following holds: The zero set of the section $\mathcal{VW}(\tau^1,\tau^2,\tau^3,\theta,\gamma,\cdot)$ in $\mathcal{C}_k^\diamond(P)$ is regular and the moduli space $\widetilde{\mathcal{M}}_{P,VW}^*(\tau^1,\tau^2,\tau^3,\theta,\gamma):=\widetilde{\mathcal{M}}_{P,VW}(\tau^1,\tau^2,\tau^3,\theta,\gamma)\cap\mathcal{B}_k^\diamond(P)=\big(\mathcal{VW}(\tau^1,\tau^2,\tau^3,\theta,\gamma,\cdot)^{-1}(0)/\mathcal{G}_{k+1}(P)\big)\cap\mathcal{B}_k^\diamond(P)$ is a smooth manifold of dimension zero.
	\end{theorem}
	
	\begin{remark}Although we have established the transversalty, the construction of the Vafa-Witten invariants is still challenging. One of the main difficulties is that, unlike the Seiberg-Witten cases, we can't get the compactness of the Vafa-Witten moduli spaces because there's no a priori estimate for the additional section $B$. By attaching some additional condition to the curvature $F_A$ of the solutions $(A,B)$ to the original Vafa-Witten equation \eqref{ac}, Tanaka gets a compactness theorem of the moduli spaces \cite[Theorem 1.3]{Tan2}.
	\end{remark}
	
	\begin{remark}
		The celebrated orientation results of gauge-theoretic moduli spaces found by Joyce, Tanaka and Upmeier show that the Vafa-Witten moduli spaces $\widetilde{\mathcal{M}}_{P,VW}^*(\tau^1,\tau^2,\tau^3,\theta,\gamma)$ always have canonical orientations \cite[Theorem 4.8]{JTU}, an orientation of $\widetilde{\mathcal{M}}_{P,VW}^*(\tau^1,\tau^2,\tau^3,\theta,\gamma)$ consists of attaching a sign $\pm 1$ to each point of it due to its dimension is zero. 
		
		In the case of zero dimensional Seiberg-Witten moduli spaces, the invariants are defined by counting the points of moduli spaces with signs. But for Vafa-Witten moduli spaces, the summation of signed points is either zero or infinity due to the lack of compactness, which is not a well defined invariant. The compactification of the Vafa-Witten moduli spaces is essential to the construction of the invariants.
	\end{remark}
	
	\ni\textbf{Acknowledgements.} 
	This research is partially supported by NSFC grants 12201255. It's a pleasure to thank Huijun Fan, Bo Dai and Yuuji Tanaka for their valuable advice and development during the research. Also, I thank my supervisor Kung-Ching Chang for his support and encouragement, and discussing with him makes me learn much about mathematical physics and this subject.
	
	\section*{Appendix}
	\appendix
	\setcounter{theorem}{0}
	\renewcommand\thetheorem{A.\arabic{theorem}}
	
	In this appendix, we list the notations and the proofs of lemmas used in this paper.
	
	Let $V$ be a finite dimensional inner product space. Let $\{ e_1,\cdots, e_n\}$ be an orthonormal basis for $V$, and  $\{e^1,\cdots,e^n\}$  the dual basis for $V^*$. For any $\alpha\in\Lambda^{p}V^*$, $\beta\in\Lambda^{q}V^*$, we define
	$$\alpha\centerdot\beta=(-1)^{p-1}\sum_{i=1}^n(\iota_{e_i}\alpha)\wedge(\iota_{e_i}\beta)\in\Lambda^{p+q-2}V^*,$$
	where $\iota_{e_i}$ is the contraction with $e_i$.
	
	The exterior algebra $\L^\bullet V^*$ inherits an inner product, such that $\{ e^{i_1}\wedge e^{i_2}\wedge\cdots\wedge e^{i_{p}}; 0\leq p\leq n, i_1 <i_2 <\cdots <i_p\}$ forms an orthonormal basis.  The inner product of two elemmaents $\a,\b \in \L^\bullet V^*$ is denoted as $\a\cdot\b$,

	Replacing $V$ by the tangent space $T_xX$ at any point $x\in X$ of a Riemannian manifold $X$, we can define the products $\centerdot$ and $\cdot$ for differential forms. Let $\mathfrak{g}$ be the  Lie algebra of a compact Lie group $G$, equipped with an invariant inner product $\langle\cdot,\cdot\rangle$, here invariance means $\langle [\xi,\eta],\zeta\rangle=\langle \xi,[\eta,\zeta]\rangle$, $\forall \xi,\eta,\zeta \in\mathfrak{g}$. If $\mathfrak{g}$ is $\mathfrak{su}(2)$, an invariant inner product is given by $\langle \xi,\eta\rangle=-\frac{1}{2}\mathrm{tr}(\xi\eta)$ \cite[(A.18)]{Ma}, where $\xi,\eta\in\mathfrak{su}(2)$ are regarded as matrices.
	
	Let $P\to X$ be a principal $G$-bundle and $\fg_P$ be the adjoint bundle. For $\fg_P$-valued forms, for example, if $\omega_1=\xi_1\otimes\alpha_1$ and $\omega_2=\xi_2\otimes\a_2$, where $\xi_1,\xi_2\in\W^0(X, \fg_P)$ and $\alpha_1\in\Omega^p(X)$, $\a_2\in\Omega^q(X)$, we define $[\omega_1\centerdot\omega_2] =[\xi_1,\xi_2]\otimes\alpha_1\centerdot\a_2$ and $[\omega_1\cdot\omega_2]=[\xi_1,\xi_2]\otimes\alpha_1\cdot\a_2$, $\langle\omega_1\cdot \omega_2\rangle=\langle \xi_1, \xi_2\rangle \a_1\cdot \a_2$. See  \cite[Appendix]{Ma} for more detail.
	
	On a Riemannian 4-manifold $X$, the \emph{rank} of a section $B\in\Omega^{2,+}(X, \fg_P)$ is defined as follows. Choose local frames for $\fg_P$ and  $\Lambda^{2,+}(T^*X)$, then the section $B$ is represented by a $d\times 3$ matrix-valued function with respect to the local frames, where $d=\dim G$. The rank of $B$ at a point of $X$ is just the rank of the matrix at that point, and $\rank(B)$ is the maximum of the pointwise rank over $X$. If $\mathfrak{g}$ is $\mathfrak{su}(2)$, then the maximum of the rank of $B$ is $d=3$.
	
	\begin{lemma}\label{xx}Let $\alpha$, $\beta$ be two elements of $\mathfrak{su}(2)$ and they are linearly independent, then $\alpha,\beta,[\alpha,\beta]$ form a basis of the linear space $\mathfrak{su}(2)$.
	\end{lemma}
	
	\begin{proof} Let $\alpha=\alpha_1\eta_1+\alpha_2\eta_2+\alpha_3\eta_3$ and $\beta=\beta_1\eta_1+\beta_2\eta_2+\beta_3\eta_3$ where $\alpha_i,\beta_i\in\mathbb{R}$, $i=1,2,3$, $\{\eta_1,\eta_2,\eta_3\}$ is a  basis for $\mathfrak{su}(2)$ such that $[\eta_1,\eta_2]=2\eta_3$ and cyclic permutations (The quaternion $\textbf{i},\textbf{j},\textbf{k}$ satisfy these equalities.). Then we have
		$$[\alpha,\beta]=2(\alpha_2\beta_3-\alpha_3\beta_2)\eta_1+2(\alpha_3\beta_1-\alpha_1\beta_3)\eta_2+2(\alpha_1\beta_2-\alpha_2\beta_1)\eta_3.$$
		We need to show that $\alpha,\beta,[\alpha,\beta]$ are linearly independent. Let $k_1,k_2,k_3\in\mathbb{R}$ be three real numbers and $$k_1\alpha+k_2\beta+k_3[\alpha,\beta]=0,$$
		then note that $\{\eta_1,\eta_2,\eta_3\}$ are linearly independent, we must have
		$$\left\{
		\begin{aligned}
			k_1\alpha_1+k_2\beta_1+2k_3(\alpha_2\beta_3-\alpha_3\beta_2)=0\\
			k_1\alpha_2+k_2\beta_2+2k_3(\alpha_3\beta_1-\alpha_1\beta_3)=0\\
			k_1\alpha_3+k_2\beta_3+2k_3(\alpha_1\beta_2-\alpha_2\beta_1)=0.\\
		\end{aligned}\right.$$
		The determinant of the above linear system of equations is
		$$\begin{vmatrix}
			
			\alpha_1 & \beta_1 & 2(\alpha_2\beta_3-\alpha_3\beta_2)\\
			
			\alpha_2 & \beta_2 & 2(\alpha_3\beta_1-\alpha_1\beta_3)\\
			
			\alpha_3 & \beta_3 & 2(\alpha_1\beta_2-\alpha_2\beta_1)\\
			
		\end{vmatrix}=2(\alpha_2\beta_3-\alpha_3\beta_2)^2+2(\alpha_3\beta_1-\alpha_1\beta_3)^2+2(\alpha_1\beta_2-\alpha_2\beta_1)^2.$$
		Note that $\alpha$, $\beta$ are linearly independent, so $\alpha\neq 0$, $\beta\neq 0$ and there's not a nonzero real number $k\in\mathbb{R}$ such that $\alpha=k\beta$, which means $\alpha_2\beta_3-\alpha_3\beta_2,\alpha_3\beta_1-\alpha_1\beta_3,\alpha_1\beta_2-\alpha_2\beta_1$ cannot be all zero, so the determinant is greater than zero, which implies $k_1=k_2=k_3=0$, so $\alpha,\beta,[\alpha,\beta]$ are linearly independent and they form a basis of the three-dimensional linear space $\mathfrak{su}(2)$.
	\end{proof}

\begin{lemma}\label{xxx}Let $\zeta\in\mathfrak{su}(2)\otimes\Lambda^{1}\mathbb{R}^4$, $\nu\in\Lambda^{2,+}\mathbb{R}^4$ and $\zeta\centerdot\nu=\zeta$, then $\zeta=0$.
\end{lemma}
\begin{proof}
	The equality $\zeta\centerdot\nu=\zeta$ can be reduced to the same form of real coefficients, so to prove the lemma, we can assume $\zeta\in\Lambda^{1}\mathbb{R}^4$.
	
	Choose an oriented orthonormal basis $\{e^1,e^2,e^3,e^4\}$ for $\mathbb{R}^4$, let $\zeta=\zeta_1e^1+\zeta_2e^2+\zeta_3e^3+\zeta_4e^4$ and $\nu=\nu_1(e^1\wedge e^2+e^3\wedge e^4)+\nu_2(e^1\wedge e^3+e^4\wedge e^2)+\nu_3(e^1\wedge e^4+e^2\wedge e^3)$, $\zeta_i,\nu_j\in\R$, $i=1,2,3,4$, $j=1,2,3$. Then $\zeta\centerdot\nu=\zeta$ implies
	\begin{align*}
		\zeta_1e^1+\zeta_2e^2+\zeta_3e^3+\zeta_4e^4=&(\zeta_1e^1+\zeta_2e^2+\zeta_3e^3+\zeta_4e^4)\centerdot(\nu_1(e^1\wedge e^2+e^3\wedge e^4)\\
		&+\nu_2(e^1\wedge e^3+e^4\wedge e^2)+\nu_3(e^1\wedge e^4+e^2\wedge e^3))\\
		=&\zeta_1(\nu_1e^2+\nu_2e^3+\nu_3e^4)+\zeta_2(-\nu_1e^1-\nu_2e^4+\nu_3e^3)+\zeta_3(\nu_1e^4-\nu_2e^1-\nu_3e^2)\\
		&+\zeta_4(-\nu_1e^3+\nu_2e^2-\nu_3e^1)\\
		=&(-\zeta_2\nu_1-\zeta_3\nu_2-\zeta_4\nu_3)e^1+(\zeta_1\nu_1-\zeta_3\nu_3+\zeta_4\nu_2)e^2+(\zeta_1\nu_2+\zeta_2\nu_3-\zeta_4\nu_1)e^3\\
		&+(\zeta_1\nu_3-\zeta_2\nu_2+\zeta_3\nu_1)e^4
	\end{align*}
Compare the coefficients of $\{e^1,e^2,e^3,e^4\}$ on both sides, we have
$$\left\{
\begin{aligned}
	\zeta_1+\zeta_2\nu_1+\zeta_3\nu_2+\zeta_4\nu_3&=0 \\			
	-\zeta_1\nu_1+\zeta_2+\zeta_3\nu_3-\zeta_4\nu_2&=0 \\	
	-\zeta_1\nu_2-\zeta_2\nu_3+\zeta_3+\zeta_4\nu_1&=0 \\	
	-\zeta_1\nu_3+\zeta_2\nu_2-\zeta_3\nu_1+\zeta_4&=0 	
\end{aligned}
\right.$$
Consider $\zeta_1,\zeta_2,\zeta_3,\zeta_4$ as unknowns, the determinant of the above system of linear equations is
$$		D_{\omega_1}=	\left|\begin{array}{cccc}
	1 & \nu_1 & \nu_2 & \nu_3\\
	-\nu_1 & 1 & \nu_3 & -\nu_2\\
	-\nu_2 & -\nu_3 & 1 & \nu_1\\
	-\nu_3 & \nu_2 & -\nu_1 & 1
\end{array}\right|=\left(1+\nu_1^2+\nu_2^2+\nu_3^2\right)^2>0,$$
hence $\zeta_1,\zeta_2,\zeta_3,\zeta_4$ are all zero and $\zeta=0$.
\end{proof}

	\begin{lemma}\label{xxy}Let $B\in\mathfrak{su}(2)\otimes\Lambda^{2,+}\mathbb{R}^4$, $C\in\mathfrak{su}(2)$ and $[B,C]\not=0$, $0\not=\theta\in\Lambda^{1}\mathbb{R}^4$, then $(B+[B,C])\centerdot\theta+C\otimes\theta$ is rank 3.
	\end{lemma}
	\begin{proof}
		According to the singular value decomposition for $\mathfrak{su}(2)\otimes\Lambda^{2,+}\mathbb{R}^4$ in  \cite[\S 4.1.1]{Ma}, there exist oriented orthonormal basis $\{e^1,e^2,e^3,e^4\}$ for $\mathbb{R}^4$ and a basis $\{\eta_1,\eta_2,\eta_3\}$ for $\mathfrak{su}(2)$   such that $[\eta_1,\eta_2]=2\eta_3$ and cyclic permutations, $B=B_{1}\eta_1\otimes(e^1\wedge e^2+e^3\wedge e^4)+B_{2}\eta_2\otimes(e^1\wedge e^3+e^4\wedge e^2)+B_{3}\eta_3\otimes(e^1\wedge e^4+e^2\wedge e^3)$, 
		where $B_{1},B_{2},B_{3}\in\R$. Let $\theta=\theta_1e^1+\theta_2e^2+\theta_3e^3+\theta_4e^4$ where $\theta_k\in\R$, $k=1,2,3,4$. $C=C_1\eta_1+C_2\eta_2+C_3\eta_3$, $C_{1},C_{2},C_{3}\in\R$. Then $0\not=[B,C]=[B_{1}\eta_1,C]\otimes(e^1\wedge e^2+e^3\wedge e^4)+[B_{2}\eta_2,C]\otimes(e^1\wedge e^3+e^4\wedge e^2)+[B_{3}\eta_3,C]\otimes(e^1\wedge e^4+e^2\wedge e^3)$ implies one of $[B_1\eta_1,C],[B_2\eta_2,C],[B_3\eta_3,C]$ is nonzero. Without loss of generality, let $0\not=[B_1\eta_1,C]=2B_1C_2\eta_3-2B_1C_3\eta_2$, which means $B_1\neq 0$ and $C_2^2+C_3^2>0$. Note that
		\begin{align*}
			&(B+[B,C])\centerdot\theta+C\otimes\theta\\
			=&-(B_1\eta_1+[B_1\eta_1,C_1\eta_1+C_2\eta_2+C_3\eta_3])\theta_1\otimes e^2-(B_2\eta_2+[B_2\eta_2,C_1\eta_1+C_2\eta_2+C_3\eta_3])\theta_1\otimes e^3\\
			&-(B_3\eta_3+[B_3\eta_3,C_1\eta_1+C_2\eta_2+C_3\eta_3])\theta_1\otimes e^4+(B_1\eta_1+[B_1\eta_1,C_1\eta_1+C_2\eta_2+C_3\eta_3])\theta_2\otimes e^1\\
			&+(B_2\eta_2+[B_2\eta_2,C_1\eta_1+C_2\eta_2+C_3\eta_3])\theta_2\otimes e^4
			-(B_3\eta_3+[B_3\eta_3,C_1\eta_1+C_2\eta_2+C_3\eta_3])\theta_2\otimes e^3\\
			&-(B_1\eta_1+[B_1\eta_1,C_1\eta_1+C_2\eta_2+C_3\eta_3])\theta_3\otimes e^4
			+(B_2\eta_2+[B_2\eta_2,C_1\eta_1+C_2\eta_2+C_3\eta_3])\theta_3\otimes e^1\\
			&+(B_3\eta_3+[B_3\eta_3,C_1\eta_1+C_2\eta_2+C_3\eta_3])\theta_3\otimes e^2
			+(B_1\eta_1+[B_1\eta_1,C_1\eta_1+C_2\eta_2+C_3\eta_3])\theta_4\otimes e^3\\
			&-(B_2\eta_2+[B_2\eta_2,C_1\eta_1+C_2\eta_2+C_3\eta_3])\theta_4\otimes e^2+(B_3\eta_3+[B_3\eta_3,C_1\eta_1+C_2\eta_2+C_3\eta_3])\theta_4\otimes e^1\\
			&+(C_1\eta_1+C_2\eta_2+C_3\eta_3)\theta_1\otimes e^1+(C_1\eta_1+C_2\eta_2+C_3\eta_3)\theta_2\otimes e^2+(C_1\eta_1+C_2\eta_2+C_3\eta_3)\theta_3\otimes e^3\\
			&+(C_1\eta_1+C_2\eta_2+C_3\eta_3)\theta_4\otimes e^4.\\
			=&-B_1(\eta_1+2C_2\eta_3-2C_3\eta_2)\theta_1\otimes e^2-B_2(\eta_2-2C_1\eta_3+2C_3\eta_1)\theta_1\otimes e^3-B_3(\eta_3+2C_1\eta_2-2C_2\eta_1)\theta_1\otimes e^4\\
			&+B_1(\eta_1+2C_2\eta_3-2C_3\eta_2)\theta_2\otimes e^1+B_2(\eta_2-2C_1\eta_3+2C_3\eta_1)\theta_2\otimes e^4-B_3(\eta_3+2C_1\eta_2-2C_2\eta_1)\theta_2\otimes e^3\\
			&-B_1(\eta_1+2C_2\eta_3-2C_3\eta_2)\theta_3\otimes e^4+B_2(\eta_2-2C_1\eta_3+2C_3\eta_1)\theta_3\otimes e^1+B_3(\eta_3+2C_1\eta_2-2C_2\eta_1)\theta_3\otimes e^2\\
			&+B_1(\eta_1+2C_2\eta_3-2C_3\eta_2)\theta_4\otimes e^3-B_2(\eta_2-2C_1\eta_3+2C_3\eta_1)\theta_4\otimes e^2+B_3(\eta_3+2C_1\eta_2-2C_2\eta_1)\theta_4\otimes e^1\\
			&+(C_1\eta_1+C_2\eta_2+C_3\eta_3)\theta_1\otimes e^1+(C_1\eta_1+C_2\eta_2+C_3\eta_3)\theta_2\otimes e^2+(C_1\eta_1+C_2\eta_2+C_3\eta_3)\theta_3\otimes e^3\\
			&+(C_1\eta_1+C_2\eta_2+C_3\eta_3)\theta_4\otimes e^4.\\
			=&\eta_1\big((B_1\theta_2+2B_2C_3\theta_3-2B_3C_2\theta_4+C_1\theta_1)\otimes e^1+(-B_1\theta_1-2B_3C_2\theta_3-2B_2C_3\theta_4+C_1\theta_2)\otimes e^2\\
			&+(B_1\theta_4-2B_2C_3\theta_1+2B_3C_2\theta_2+C_1\theta_3)\otimes e^3+(-B_1\theta_3+2B_3C_2\theta_1+2B_2C_3\theta_2+C_1\theta_4)\otimes e^4\big)\\
			&\eta_2\big((B_2\theta_3-2B_1C_3\theta_2+2B_3C_1\theta_4+C_2\theta_1)\otimes e^1+(-B_2\theta_4+2B_1C_3\theta_1+2B_3C_1\theta_3+C_2\theta_2)\otimes e^2\\
			&+(-B_2\theta_1-2B_1C_3\theta_4-2B_3C_1\theta_2+C_2\theta_3)\otimes e^3+(B_2\theta_2+2B_1C_3\theta_3-2B_3C_1\theta_1+C_2\theta_4)\otimes e^4\big)\\
			&\eta_3\big((B_3\theta_4-2B_2C_1\theta_3+2B_1C_2\theta_2+C_3\theta_1)\otimes e^1+(B_3\theta_3+2B_2C_1\theta_4-2B_1C_2\theta_1+C_3\theta_2)\otimes e^2\\
			&+(-B_3\theta_2+2B_2C_1\theta_1+2B_1C_2\theta_4+C_3\theta_3)\otimes e^3+(-B_3\theta_1-2B_2C_1\theta_2-2B_1C_2\theta_3+C_3\theta_4)\otimes e^4\big)\\
			:=&\eta_1\otimes\omega_1+\eta_2\otimes\omega_2+\eta_3\otimes\omega_3
		\end{align*}
		for some $\omega_1,\omega_2,\omega_3\in\Lambda^{1}\mathbb{R}^4$.  If  the rank of $(B+[B,C])\centerdot\theta+C\otimes\theta$ is not 3, then one of $\omega_1,\omega_2,\omega_3$ is zero. $\omega_1=0$ means
		$$\left\{
		\begin{aligned}
			B_1\theta_2+2B_2C_3\theta_3-2B_3C_2\theta_4+C_1\theta_1&=0 \\			-B_1\theta_1-2B_3C_2\theta_3-2B_2C_3\theta_4+C_1\theta_2&=0 \\
			B_1\theta_4-2B_2C_3\theta_1+2B_3C_2\theta_2+C_1\theta_3&=0 \\
			-B_1\theta_3+2B_3C_2\theta_1+2B_2C_3\theta_2+C_1\theta_4&=0 
		\end{aligned}
		\right.$$
		Consider $\theta_1,\theta_2,\theta_3,\theta_4$ as unknowns and note that $B_1\neq 0$, the determinant of the above system of linear equations is 
		$$		D_{\omega_1}=	\left|\begin{array}{cccc}
				C_1 & B_1 & 2B_2C_3 & -2B_3C_2\\
				-B_1 & C_1 & -2B_3C_2 & -2B_2C_3\\
				-2B_2C_3 & 2B_3C_2 & C_1 & B_1\\
				2B_3C_2 & 2B_2C_3 & -B_1 & C_1
			\end{array}\right|=\left(B_1^2+C_1^2+4 B_2^2 C_3^2+4 B_3^2 C_2^2\right)^2>0,$$
		contradicts to $\theta\neq 0$, so $\omega_1\not=0$. Similarly, the corresponding determinant of  $\omega_2$ and $\omega_3$ are
		$$D_{\omega_2}=\left(B_2^2+C_2^2+4 B_3^2 C_1^2+4 B_1^2 C_3^2\right)^2$$ and $$D_{\omega_3}=\left(B_3^2+C_3^2+4 B_1^2 C_2^2+4 B_2^2 C_1^2\right)^2.$$
		If $C_2=0$, then $C_3\neq 0$, $4B_1^2C_3^2>0$, we have $D_{\omega_2}\geq (4B_1^2C_3^2)^2>0$ and $D_{\omega_2}\geq (C_3^2)^2>0$, which means $\omega_2\not=0$ and $\omega_3\not=0$, so $(B+[B,C])\centerdot\theta+C\otimes\theta$ is rank 3. The case of $C_3 = 0$ can be discussed similarly.
		
		Hence we must have $(B+[B,C])\centerdot\theta+C\otimes\theta$  rank 3.
	\end{proof}

\end{document}